\documentclass{article}
\usepackage{amsfonts}

\newtheorem{theorem}{Theorem}

\newtheorem{corollary}[theorem]{Corollary}

\newtheorem{definition}[theorem]{Definition}
\newtheorem{example}[theorem]{Example}

\newtheorem{proposition}[theorem]{Proposition}
\newtheorem{remark}[theorem]{Remark}

\newenvironment{proof}[1][Proof]{\noindent\textbf{#1.} }{\ \rule{0.5em}{0.5em}}
\input{tcilatex}
\begin{document}

\title{\textbf{Fuzzy Set with Hyperbolic Valued Membership Function}}
\author{Chinmay Ghosh$^{1}$, Sanjib Kumar Datta$^{2}$, Soumen Mondal$^{3}$%
\qquad \\
$^{1}$Department of Mathematics\\
Kazi Nazrul University\\
Nazrul Road, P.O.- Kalla C.H.\\
Asansol-713340, West Bengal, India \\
chinmayarp@gmail.com \\
$^{2}$Department of Mathematics\\
University of Kalyani\\
P.O.-Kalyani, Dist-Nadia, PIN-741235,\\
West Bengal, India\\
sanjibdatta05@gmail.com\\
$^{3}$28, Dolua Dakshinpara Haridas Primary School\\
Beldanga, Murshidabad\\
Pin-742133\\
West Bengal, India\\
mondalsoumen79@gmail.com}
\date{}
\maketitle

\begin{abstract}
In this article, we have introdued $\mathbb{D}$-fuzzy sets. We have
discussed the notions of inclusion, union, intersection, complementation and
convexity of such $\mathbb{D}$-fuzzy sets. Also we have proved separation
theorem of convex $\mathbb{D}$-fuzzy sets. 

\end{abstract}

\section{\qquad Introduction and Preliminaries}

Let $%
\mathbb{R}
$ be the field of real numbers. Then $%
\mathbb{R}
^{2}=\left\{ \left( x,y\right) :x,y\in 
\mathbb{R}
\right\} $ forms a two dimensional vector space over $%
\mathbb{R}
$ with usual vector addition (componentwise) and scalar multiplication. The
standard basis is $\left\{ \left( 1,0\right) ,\left( 0,1\right) \right\} .$
Also the field structure of complex numbers $%
\mathbb{C}
$ can be formed from $%
\mathbb{R}
^{2}$ with respect to the said addition and the multiplication defined by%
\[
\left( x_{1},y_{1}\right) .\left( x_{2},y_{2}\right) =\left(
x_{1}x_{2}-y_{1}y_{2},x_{1}y_{2}+y_{1}x_{2}\right) . 
\]%
Denoting $\left( 1,0\right) ,\left( 0,1\right) $ by $\mathbf{1}$ and $%
\mathbf{i}$ respectively\textbf{,} we can represent any element $\left(
x,y\right) $ of this field by $x.\mathbf{1}+y.\mathbf{i,}$ and we write 
\[
\mathbb{C}
=\{x.\mathbf{1}+y.\mathbf{i}:x,y\in 
\mathbb{R}
\}. 
\]

If no confusion arises we may write $x+y\mathbf{i}$ instead of $x.\mathbf{1}%
+y.\mathbf{i.}$ The elements of $%
\mathbb{C}
$ describe the two dimensional Euclidean geometry.

Note that 
\[
\mathbf{i}^{2}=\mathbf{i.i=}\left( 0,1\right) .\left( 0,1\right) =\left(
-1,0\right) =-\left( 1,0\right) =-\mathbf{1.} 
\]

Another ring structure may be found from $%
\mathbb{R}
^{2}$ by redefining the multiplication as 
\[
\left( x_{1},y_{1}\right) .\left( x_{2},y_{2}\right) =\left(
x_{1}x_{2}+y_{1}y_{2},x_{1}y_{2}+y_{1}x_{2}\right) . 
\]

In this ring we denote $\left( 1,0\right) ,\left( 0,1\right) $ by $\mathbf{1}
$ and $\mathbf{k}$ respectively.

Then 
\[
\mathbf{k}^{2}=\mathbf{k.k=}\left( 0,1\right) .\left( 0,1\right) =\left(
1,0\right) =\mathbf{1.} 
\]

Though this ring is commutative and contains multiplicative identity, it has
divisors of zero. So it fails to be an integral domain. Any element $\left(
x,y\right) $ of this ring can be represented by $x.\mathbf{1}+y.\mathbf{k}$
or $x+y\mathbf{k.}$ The elements are called hyperbolic numbers and the ring
of all hyperbolic numbers is denoted by $\mathbb{D}$, i.e.,%
\[
\mathbb{D}=\{x+y\mathbf{k}:x,y\in 
\mathbb{R}
\}. 
\]

Henceforth we will write $1$ and $0$ as the elements $\mathbf{1=}\left(
1,0\right) $ and $\mathbf{0=}\left( 0,0\right) $ in $\mathbb{D}$
respectively. As $%
\mathbb{C}
$ describes two dimensional Euclidean geometry, $\mathbb{D}$ also describes
two dimensional Minkowski space-time (Lorentzian) geometry \cite{Cat 1}, 
\cite{Cat 2}.

The two hyperbolic numbers (zero divisors)%
\[
\mathbf{e}_{1}=\frac{1+\mathbf{k}}{2},\mathbf{e}_{2}=\frac{1-\mathbf{k}}{2} 
\]%
satisfy 
\[
\mathbf{e}_{1}^{2}=\mathbf{e}_{1},\mathbf{e}_{2}^{2}=\mathbf{e}_{2},\mathbf{e%
}_{1}+\mathbf{e}_{2}=1,\mathbf{e}_{1}\mathbf{e}_{2}=0. 
\]

Further an arbitrary element $\alpha =a_{1}+a_{2}\mathbf{k}\in \mathbb{D}$
can be uniquely represented as 
\[
\mathbb{\alpha =(}a_{1}+a_{2})\mathbf{e}_{1}+\mathbb{(}a_{1}-a_{2})\mathbf{e}%
_{2}, 
\]%
and hence $\left\{ \mathbf{e}_{1},\mathbf{e}_{2}\right\} $ forms an
idempotent basis in $\mathbb{D}.$

Sum and the product of two hyperbolic numbers can be defined pointwise with
the idempotent basis.

We say that a hyperbolic number $\alpha =\alpha _{1}\mathbf{e}_{1}+\alpha
_{2}\mathbf{e}_{2}\in \mathbb{D}$ is positive if $\alpha _{1},\alpha _{2}>0.$
Thus the set of positive hyperbolic numbers $\mathbb{D}^{+}$ is given by%
\begin{eqnarray*}
\mathbb{D}^{+} &\mathbb{=}&\mathbb{\{}\alpha =\alpha _{1}\mathbf{e}%
_{1}+\alpha _{2}\mathbf{e}_{2}:\alpha _{1},\alpha _{2}>0\mathbb{\}} \\
&=&\mathbb{\{}\alpha =a_{1}+a_{2}\mathbf{k}:a_{1}>\left\vert
a_{2}\right\vert \mathbb{\}}.
\end{eqnarray*}

Let $\mathbb{O}$ be the set of all zero divisors in $\mathbb{D}$ and $%
\mathbb{O}_{0}=\mathbb{O}\cup \{\mathbf{0}\}.$\ We use the notation $\mathbb{%
D}_{0}^{+}=\mathbb{D}^{+}\cup \mathbb{O}_{0}.$

For $\alpha ,\beta \in \mathbb{D},~$define~a relation $\preceq $ on $\mathbb{%
D}$~\cite{Lun}$\ $by $\alpha \preceq \beta $ whenever $\beta -\alpha \in 
\mathbb{D}_{0}^{+}.$ This relation is reflexive, anti-symmetric as well as
transitive and hence defines a partial order relation on $\mathbb{D}.$ If we
write the hyperbolic numbers $\alpha ,\beta $ in idempotent representation
as $\mathbb{\alpha =}\alpha _{1}\mathbf{e}_{1}+\alpha _{2}\mathbf{e}_{2}$
and $\mathbb{\beta =\beta }_{1}\mathbf{e}_{1}+\mathbb{\beta }_{2}\mathbf{e}%
_{2},$ then $\alpha \preceq \beta $ implies that $\alpha _{1}\leq \beta _{1}$
and $\alpha _{2}\leq \beta _{2}.$ And by $\alpha \prec \beta $ we mean $%
\alpha _{1}<\beta _{1}$ and $\alpha _{2}<\beta _{2}.$ Two hyperbolic numbers 
$\alpha ,\beta ~$are said to be comparable if either $\alpha \prec \beta $
or $\alpha =\beta $ or $\beta \prec \alpha .$ We write $\max\nolimits_{%
\mathbb{D}}\left\{ \alpha ,\beta \right\} =\alpha $ if $\beta \preceq \alpha 
$ and $\max\nolimits_{\mathbb{D}}\left\{ \alpha ,\beta \right\} =\beta $ if $%
\alpha \preceq \beta .$ Similarly $\min_{\mathbb{D}}\left\{ \alpha ,\beta
\right\} =\alpha $ if $\alpha \preceq \beta $ and $\min_{\mathbb{D}}\left\{
\alpha ,\beta \right\} =\beta $ if $\beta \preceq \alpha .$ For more than
two hyperbolic numbers $\max\nolimits_{\mathbb{D}}$ and $\min_{\mathbb{D}}$
are defined only if any pair of the given hyperbolic numbers are comparable.

Let $\alpha =\alpha _{1}\mathbf{e}_{1}+\alpha _{2}\mathbf{e}_{2},\beta =%
\mathbb{\beta }_{1}\mathbf{e}_{1}+\mathbb{\beta }_{2}\mathbf{e}_{2}\in 
\mathbb{D}$ with $\alpha \preceq \beta .$ The closed hyperbolic interval ($%
\mathbb{D}-$interval) \cite{Bal} is defined by 
\[
\left[ \alpha ,\beta \right] _{\mathbb{D}}=\left\{ \zeta \in \mathbb{D}%
:\alpha \preceq \zeta \preceq \beta \right\} . 
\]

Similarly the open hyperbolic interval ($\mathbb{D}-$interval) is defined by%
\[
\left( \alpha ,\beta \right) _{\mathbb{D}}=\left\{ \zeta \in \mathbb{D}%
:\alpha \prec \zeta \prec \beta \right\} . 
\]

The hyperbolic length of the $\mathbb{D}-$interval $\left[ \alpha ,\beta %
\right] _{\mathbb{D}}$ or $\left( \alpha ,\beta \right) _{\mathbb{D}}$\ is
defined by 
\[
l_{\mathbb{D}}\left( \left[ \alpha ,\beta \right] _{\mathbb{D}}\right)
=\beta -\alpha . 
\]

$\left[ \alpha ,\beta \right] _{\mathbb{D}}$ is called a degenerate closed $%
\mathbb{D}-$interval if $\beta -\alpha \in \mathbb{D}^{+}\cap \mathbb{O}$
and $\left[ \alpha ,\beta \right] _{\mathbb{D}}$ is called a nondegenerate
closed $\mathbb{D}-$interval if $\beta -\alpha $ $\in \mathbb{D}%
^{+}\backslash \mathbb{O}$.

The hyperbolic modulus of a hyperbolic number is defined by%
\[
\left\vert \alpha \right\vert _{k}=\left\vert \alpha _{1}\mathbf{e}%
_{1}+\alpha _{2}\mathbf{e}_{2}\right\vert _{k}=\left\vert \alpha
_{1}\right\vert \mathbf{e}_{1}+\left\vert \alpha _{2}\right\vert \mathbf{e}%
_{2}\in \mathbb{D}_{0}^{+} 
\]%
where $\left\vert .\right\vert $ is the modulus of a real number. $\alpha
=\alpha _{1}\mathbf{e}_{1}+\alpha _{2}\mathbf{e}_{2}$ is said to be nearer
to $1$ or $0$ in the hyperbolic sense if both $\alpha _{1}$ and $\alpha _{2}$
are nearer to $1$ or $0$ in the real sense respectively.

The hyperbolic valued norm \cite{Al} on a $\mathbb{D-}$module space $X$ is a
mapping $\left\Vert .\right\Vert _{\mathbb{D}}:X\rightarrow \mathbb{D}^{+}$
that satisfies the following properties:

$(i)$ $\left\Vert x\right\Vert _{\mathbb{D}}\succeq 0$ for all $x\in X$ and $%
\left\Vert x\right\Vert _{\mathbb{D}}=0$ iff $x=0\in X.$

$(ii)$ $\left\Vert \alpha x\right\Vert _{\mathbb{D}}=\left\vert \alpha
\right\vert _{k}\left\Vert x\right\Vert _{\mathbb{D}}$ for all $\alpha \in 
\mathbb{D}$ and for all $x\in X.$

$(iii)$ $\left\Vert x+y\right\Vert _{\mathbb{D}}\preceq \left\Vert
x\right\Vert _{\mathbb{D}}+\left\Vert y\right\Vert _{\mathbb{D}}$ for all $%
x,y\in X.$

A function $d_{\mathbb{D}}:X\times X\longrightarrow \mathbb{D}_{0}^{+}$
defined on a nonempty set $X$ satisfying the following:

$(i)$ $d_{\mathbb{D}}(x,y)\succeq 0$ for all $x,y\in X$ and $d_{\mathbb{D}%
}(x,y)=0$ if and only if $x=y,$

$(ii)$ $d_{\mathbb{D}}(x,y)=d_{\mathbb{D}}(y,x)$ for all $x,y\in X,$

$(iii)$ $d_{\mathbb{D}}(x,y)\preceq d_{\mathbb{D}}(x,z)+d_{\mathbb{D}}(z,y)$
for all $x,y,z\in X,$

is called a hyperbolic valued metric \cite{Gh} or $\mathbb{D-}$metric on $X.$

Being motivated from the articles \cite{Bal} and \cite{Za} we wish to find
fuzzy sets in the hyperbolic settings.

In the book of G. B. Price \cite{Pri}, the description of hyperbolic numbers
is found. But G. Sobczyk \cite{Sob} first realized the importance of such
numbers and called them sister of complex numbers. The two books \cite{Cat 1}
and \cite{Cat 2} are good references for hyperbolic numbers. The readers of
this paper may go through the books \cite{alp}, \cite{Lun} and \cite{Ola}.
Recently many researchers (\cite{Bal}, \cite{Gh}, \cite{Sai}, \cite{Tell})
are investigating different aspets of hyperbolic numbers.

\section{Main Results}

Let $X$ be a set of points and $x\in X.$ A $\mathbb{D}$-fuzzy set $A$ in $X$
is characterized by a hyperbolic valued membership function ($\mathbb{D}$%
-membership function) $f_{A}^{\mathbb{D}}:$ $X\rightarrow \lbrack 0,1]_{%
\mathbb{D}}$ which associates with each point in $X$ a hyperbolic number in
the hyperbolic interval $[0,1]_{\mathbb{D}},$ with the value of $f_{A}^{%
\mathbb{D}}(x)$ at $x$ representing the "$\mathbb{D}$-grade of $\mathbb{D}$%
-membership" of $x$ in $A.$ Thus the nearer the value of $f_{A}^{\mathbb{D}%
}(x)$ to $1$, the higher $\mathbb{D}$-grade of membership of $x$ in $A.$

A $\mathbb{D}$-fuzzy set $A$ in $X$ with $\mathbb{D}$-membership function $%
f_{A}^{\mathbb{D}}(x),$ can be written as 
\[
A=\{(x,f_{A}^{\mathbb{D}}(x)):x\in X\}. 
\]

Let $f_{A}^{\mathbb{D}}(x)=f_{A}^{1}(x)\mathbf{e}_{1}+f_{A}^{2}(x)\mathbf{e}%
_{2}$ be the idempotent representation of $f_{A}^{\mathbb{D}}(x).$ Then $%
A_{1}=\{(x,f_{A}^{1}(x):x\in X\}$ and $A_{2}=\{(x,f_{A}^{2}(x):x\in X\}$ are
two ordinary fuzzy sets.

Note that if $A$ is a set in the ordinary sense, its $\mathbb{D}$-membership
function is of the form%
\[
f_{A}^{\mathbb{D}}(x)=\left\{ 
\begin{array}{c}
1\text{ if }x\in A \\ 
0\text{ if }x\notin A%
\end{array}%
\right. 
\]

\begin{example}
Let $X=\mathbb{D}$ and $A=\{x\in \mathbb{D}$: $x\succ 1\}$ and $f_{A}^{%
\mathbb{D}}(x):X\rightarrow \lbrack 0,1]_{\mathbb{D}}$ be a function.
Representative values of such function might be : $f_{A}^{\mathbb{D}}(%
\mathbf{e}_{1})=0,$ $f_{A}^{\mathbb{D}}(0)=0,$ $f_{A}^{\mathbb{D}}(5\mathbf{e%
}_{1}+2\mathbf{e}_{2})=0.06\mathbf{e}_{1}+0.04\mathbf{e}_{2},$ $f_{A}^{%
\mathbb{D}}(5)=0.06\mathbf{e}_{1}+0.07\mathbf{e}_{2},$ $f_{A}^{\mathbb{D}%
}(200)=0.4\mathbf{e}_{1}+0.5\mathbf{e}_{2}$ etc.
\end{example}

A $\mathbb{D}$-fuzzy set is empty if and only if its $\mathbb{D}$-membership
function is zero on $X.$ Two $\mathbb{D}$-fuzzy sets $A$ and $B$ are equal
if and only if $f_{A}^{\mathbb{D}}(x)=f_{B}^{\mathbb{D}}(x)$ for all $x\in
X. $ The complement of a $\mathbb{D}$-fuzzy set $A$ is denoted by $A^{\prime
}$ and is defined by 
\[
f_{A^{^{\prime }}}^{\mathbb{D}}=1-f_{A}^{\mathbb{D}}. 
\]

Let $A$ and $B$ be two $\mathbb{D}$-fuzzy sets. Then $A$ is contained in $B$
if and only if $f_{A}^{\mathbb{D}}\preceq f_{B}^{\mathbb{D}},$ i.e.,%
\[
A\subset B\Leftrightarrow f_{A}^{\mathbb{D}}\preceq f_{B}^{\mathbb{D}}. 
\]

The union of two $\mathbb{D}$-fuzzy sets $A$ and $B$ with $\mathbb{D}$%
-membership functions $f_{A}^{\mathbb{D}}$ and $f_{B}^{\mathbb{D}}$
respectively is a $\mathbb{D}$-fuzzy set $C\,=A\cup B,$ whose $\mathbb{D}$%
-membership function is 
\[
f_{C}^{\mathbb{D}}(x)=\max\nolimits_{\mathbb{D}}\{f_{A}^{\mathbb{D}%
}(x),f_{B}^{\mathbb{D}}(x)\},\text{ }x\in X 
\]%
or, in short,%
\[
f_{C}^{\mathbb{D}}=f_{A}^{\mathbb{D}}\vee f_{B}^{\mathbb{D}}. 
\]

\begin{proposition}
The union of two $\mathbb{D}$-fuzzy sets $A$ and $B$ is the smallest $%
\mathbb{D}$-fuzzy set containing $A$ and $B.$
\end{proposition}

\begin{proof}
Let $C\,=A\cup B,$ whose $\mathbb{D}$-membership function is 
\[
f_{C}^{\mathbb{D}}(x)=\max\nolimits_{\mathbb{D}}\{f_{A}^{\mathbb{D}%
}(x),f_{B}^{\mathbb{D}}(x)\},\text{ }x\in X 
\]

Then 
\[
\max\nolimits_{\mathbb{D}}\{f_{A}^{\mathbb{D}}(x),f_{B}^{\mathbb{D}%
}(x)\}\succeq f_{A}^{\mathbb{D}}(x) 
\]

and%
\[
\max\nolimits_{\mathbb{D}}\{f_{A}^{\mathbb{D}}(x),f_{B}^{\mathbb{D}%
}(x)\}\succeq f_{B}^{\mathbb{D}}(x). 
\]

Now let $D$ be any $\mathbb{D}$-fuzzy set containing $A$ and $B.$

Then 
\[
f_{D}^{\mathbb{D}}(x)\succeq f_{A}^{\mathbb{D}}(x)\text{ , }f_{D}^{\mathbb{D}%
}(x)\succeq f_{B}^{\mathbb{D}}(x). 
\]

Hence%
\[
f_{D}^{\mathbb{D}}(x)\succeq \max\nolimits_{\mathbb{D}}\{f_{A}^{\mathbb{D}%
}(x),f_{B}^{\mathbb{D}}(x)\}=f_{C}^{\mathbb{D}}(x), 
\]

which implies that 
\[
C\subset D. 
\]
\end{proof}

$\ $The intersection of two $\mathbb{D}$-fuzzy sets $A$ and $B$ with $%
\mathbb{D}$-membership functions $f_{A}^{\mathbb{D}}$ and $f_{B}^{\mathbb{D}%
} $ respectively is a $\mathbb{D}$-fuzzy set $C\,=A\cap B,$ whose $\mathbb{D}
$-membership function is%
\[
f_{C}^{\mathbb{D}}(x)=\min\nolimits_{\mathbb{D}}\{f_{A}^{\mathbb{D}%
}(x),f_{B}^{\mathbb{D}}(x)\},\text{ }x\in X 
\]%
or%
\[
f_{C}^{\mathbb{D}}=f_{A}^{\mathbb{D}}\wedge f_{B}^{\mathbb{D}}. 
\]

As in the case of union, it is easy to prove the following proposition.

\begin{proposition}
The intersection of two $\mathbb{D}$-fuzzy sets $A$ and $B$ is the largest $%
\mathbb{D}$-fuzzy set contained in both $A$ and $B.$
\end{proposition}

Let $A,B,C$ be $\mathbb{D}$-fuzzy sets with corresponding $\mathbb{D}$%
-membership functions $f_{A}^{\mathbb{D}}$, $f_{B}^{\mathbb{D}}$ and $f_{C}^{%
\mathbb{D}}$ respectively. Then the following are satisfied:

\begin{itemize}
\item De Morgan's Laws:%
\begin{equation}
\left( A\cup B\right) ^{\prime }=A^{\prime }\cap B^{\prime }  \label{1.a}
\end{equation}%
\begin{equation}
\left( A\cap B\right) ^{\prime }=A^{\prime }\cup B^{\prime }.  \label{1.b}
\end{equation}

\item Distributive Laws:%
\begin{equation}
C\cap \left( A\cup B\right) =\left( C\cap A\right) \cup \left( C\cap B\right)
\label{2.a}
\end{equation}%
\begin{equation}
C\cup \left( A\cap B\right) =\left( C\cup A\right) \cap \left( C\cup
B\right) .  \label{2.b}
\end{equation}
\end{itemize}

Property (\ref{1.a}) can be proved by using the relation%
\[
1-\max\nolimits_{\mathbb{D}}\{f_{A}^{\mathbb{D}}(x),f_{B}^{\mathbb{D}%
}(x)\}=\min\nolimits_{\mathbb{D}}\left\{ 1-f_{A}^{\mathbb{D}}(x),1-f_{B}^{%
\mathbb{D}}(x)\right\} 
\]

which can be easily verified by testing it for the two possible cases: $%
f_{A}^{\mathbb{D}}(x)\succ f_{B}^{\mathbb{D}}(x)$ and $f_{A}^{\mathbb{D}%
}(x)\prec f_{B}^{\mathbb{D}}(x).$

Similarly we can prove property (\ref{1.b}).

Property (\ref{2.a}) can be proved by using the relation%
\[
\max\nolimits_{\mathbb{D}}\left\{ f_{C}^{\mathbb{D}}(x),\min\nolimits_{%
\mathbb{D}}\left\{ f_{A}^{\mathbb{D}}(x),f_{B}^{\mathbb{D}}(x)\right\}
\right\} =\min\nolimits_{\mathbb{D}}\left\{ \max\nolimits_{\mathbb{D}%
}\{f_{A}^{\mathbb{D}}(x),f_{C}^{\mathbb{D}}(x)\},\max\nolimits_{\mathbb{D}%
}\{f_{C}^{\mathbb{D}}(x),f_{B}^{\mathbb{D}}(x)\}\right\} 
\]

which can be verified by testing it for the six possible cases: 
\[
f_{A}^{\mathbb{D}}(x)\succ f_{B}^{\mathbb{D}}(x)\succ f_{C}^{\mathbb{D}}(x),%
\text{ }f_{A}^{\mathbb{D}}(x)\succ f_{C}^{\mathbb{D}}(x)\succ f_{B}^{\mathbb{%
D}}(x),\text{ }f_{B}^{\mathbb{D}}(x)\succ f_{A}^{\mathbb{D}}(x)\succ f_{C}^{%
\mathbb{D}}(x), 
\]%
\[
f_{B}^{\mathbb{D}}(x)\succ f_{C}^{\mathbb{D}}(x)\succ f_{A}^{\mathbb{D}}(x),%
\text{ }f_{C}^{\mathbb{D}}(x)\succ f_{A}^{\mathbb{D}}(x)\succ f_{B}^{\mathbb{%
D}}(x),\text{ }f_{C}^{\mathbb{D}}(x)\succ f_{B}^{\mathbb{D}}(x)\succ f_{A}^{%
\mathbb{D}}(x). 
\]

Similarly we can prove property (\ref{2.b}).

The algebraic sum of $\mathbb{D}$-fuzzy sets $A$ and $B,$ denoted by $A+B,$
is defined in terms of the $\mathbb{D}$-membership functions of $A$ and $B$
by%
\[
f_{A+B}^{\mathbb{D}}=f_{A}^{\mathbb{D}}+f_{B}^{\mathbb{D}} 
\]%
provided the sum $f_{A}^{\mathbb{D}}+f_{B}^{\mathbb{D}}\preceq 1.$

The algebraic product of $\mathbb{D}$-fuzzy sets $A$ and $B,$ denoted by $%
AB, $ is defined in terms of the $\mathbb{D}$-membership functions of $A$
and $B$ by%
\[
f_{AB}^{\mathbb{D}}=f_{A}^{\mathbb{D}}f_{B}^{\mathbb{D}}. 
\]

The absolute difference of $\mathbb{D}$-fuzzy sets $A$ and $B,$ denoted by $%
\left\vert A-B\right\vert _{\mathbb{D}},$ is defined in terms of the $%
\mathbb{D}$-membership functions of $A$ and $B$ by%
\[
f_{\left\vert A-B\right\vert _{\mathbb{D}}}^{\mathbb{D}}=\left\vert f_{A}^{%
\mathbb{D}}-f_{B}^{\mathbb{D}}\right\vert _{\mathbb{D}}. 
\]

Let $A$ and $B$ be two $\mathbb{D-}$fuzzy sets defined on the Universal sets 
$X$ and $Y$ with $\mathbb{D}$-membership functions $f_{A}^{\mathbb{D}}(x)$
and $f_{B}^{\mathbb{D}}(y)$ respectively$.$ Then the cartesian product of $A$
and $B,$ denoted by $A\times B,$ is defined in terms of the $\mathbb{D}$%
-membership functions of $A$ and $B$ by%
\[
f_{A\times B}^{\mathbb{D}}((x,y))=\min\nolimits_{\mathbb{D}}\{f_{A}^{\mathbb{%
D}}(x),f_{B}^{\mathbb{D}}(y)\}. 
\]

Let $A,$ $B$ and $\Lambda $ be arbitrary $\mathbb{D}$-fuzzy sets. The convex
combination of $A,$ $B$ and $\Lambda $ is denoted by $(A,$ $B;\Lambda )$ and
is defined by%
\[
(A,B;\Lambda )=\Lambda A+\Lambda ^{\prime }B 
\]

where $\Lambda ^{\prime }$ is the complement of $\Lambda .$

A basic property of the convex combination of $A,$ $B$ and $\Lambda $ is
expressed by 
\[
A\cap B\subset (A,B;\Lambda )\subset A\cup B\text{ \ \ for all }\Lambda . 
\]

This property is an immediate consequence of the inequalities%
\[
\min\nolimits_{\mathbb{D}}\left\{ f_{A}^{\mathbb{D}}(x),f_{B}^{\mathbb{D}%
}(x)\right\} \preceq \lambda f_{A}^{\mathbb{D}}(x)+(1-\lambda )f_{B}^{%
\mathbb{D}}(x)\preceq \max\nolimits_{\mathbb{D}}\left\{ f_{A}^{\mathbb{D}%
}(x),f_{B}^{\mathbb{D}}(x)\right\} ,\text{ }x\in X 
\]

which hold for all $\lambda \in \lbrack 0,1]_{\mathbb{D}}.$

A binary $\mathbb{D}$-fuzzy relation in $X$ is a $\mathbb{D}$-fuzzy set in
the product space $X\times X.$ Moreover, an $n-$ary $\mathbb{D}$-fuzzy
relation in $X$ is a $\mathbb{D}$-fuzzy set $A$ in the product space $%
X\times X\times ...\times X$ whose $\mathbb{D}$-membership function is of
the form $f_{A}^{\mathbb{D}}(x_{1},x_{2},...,x_{n}),$ where $x_{i}\in
X,i=1,2,...,n.$

Now we define convexity of $\mathbb{D-}$fuzzy sets in two different ways and
prove their equivalence.

\begin{definition}
\label{DC1} A $\mathbb{D-}$fuzzy set $A$ is convex if and only if the sets $%
\Gamma _{\alpha }$ defined by%
\[
\Gamma _{\alpha }=\{x:f_{A}^{\mathbb{D}}(x)\succeq \alpha \} 
\]

are convex for all $\alpha \in (0,1]_{\mathbb{D}}.$
\end{definition}

\begin{definition}
\label{DC2} A $\mathbb{D-}$fuzzy set $A$ is convex if and only if%
\begin{equation}
f_{A}^{\mathbb{D}}\{\lambda x_{1}+(1-\lambda )x_{2}\}\succeq \min\nolimits_{%
\mathbb{D}}\left\{ f_{A}^{\mathbb{D}}(x_{1}),f_{A}^{\mathbb{D}%
}(x_{2})\right\}  \label{DCR2}
\end{equation}

for all $x_{1},x_{2}$ in $X$ and for all $\lambda $ in $[0,1]_{\mathbb{D}}.$
\end{definition}

\begin{remark}
The above two definitions are equivalent.

Let $A$ is convex in the sense of Definition \ref{DC1} and $\alpha =f_{A}^{%
\mathbb{D}}(x_{1})\preceq f_{A}^{\mathbb{D}}(x_{2}),$ then $x_{2}\in \Gamma
_{\alpha }$ and $\lambda x_{1}+(1-\lambda )x_{2}\in \Gamma _{\alpha }$ by
the convexity of $\Gamma _{\alpha }.$ Hence%
\[
f_{A}^{\mathbb{D}}\{\lambda x_{1}+(1-\lambda )x_{2}\}\succeq \alpha =f_{A}^{%
\mathbb{D}}(x_{1})=\min\nolimits_{\mathbb{D}}\left\{ f_{A}^{\mathbb{D}%
}(x_{1}),f_{A}^{\mathbb{D}}(x_{2})\right\} . 
\]

Conversely, if $A$ is convex in the sense of Definition \ref{DC2} and $%
\alpha =f_{A}^{\mathbb{D}}(x_{1}),$ then $\Gamma _{\alpha }$ may be regarded
as the set of points $x_{2}$ for which $f_{A}^{\mathbb{D}}(x_{2})\succeq
f_{A}^{\mathbb{D}}(x_{1}).$ In virtue of (\ref{DCR2}), every point of the
form $\lambda x_{1}+(1-\lambda )x_{2},\lambda \in \lbrack 0,1]_{\mathbb{D}},$
is also in $\Gamma _{\alpha }$ and hence $\Gamma _{\alpha }$ is convex set.
\end{remark}

\begin{theorem}
If two $\mathbb{D-}$fuzzy sets $A$ and $B$ are convex then $A\cap B$ is also
convex.
\end{theorem}

\begin{proof}
Let $C=A\cap B$. Then 
\[
f_{C}^{\mathbb{D}}\{\lambda x_{1}+(1-\lambda )x_{2}\}=\min\nolimits_{\mathbb{%
D}}\{f_{A}^{\mathbb{D}}\{\lambda x_{1}+(1-\lambda )x_{2}\},f_{B}^{\mathbb{D}%
}\{\lambda x_{1}+(1-\lambda )x_{2}\}\}. 
\]

Now since $A$ and $B$ are convex%
\[
f_{A}^{\mathbb{D}}\{\lambda x_{1}+(1-\lambda )x_{2}\}\succeq \min\nolimits_{%
\mathbb{D}}\left\{ f_{A}^{\mathbb{D}}(x_{1}),f_{A}^{\mathbb{D}%
}(x_{2})\right\} 
\]%
\[
f_{B}^{\mathbb{D}}\{\lambda x_{1}+(1-\lambda )x_{2}\}\succeq \min\nolimits_{%
\mathbb{D}}\left\{ f_{B}^{\mathbb{D}}(x_{1}),f_{B}^{\mathbb{D}%
}(x_{2})\right\} 
\]

and hence%
\begin{eqnarray*}
&&f_{C}^{\mathbb{D}}\{\lambda x_{1}+(1-\lambda )x_{2}\} \\
&\succeq &\min\nolimits_{\mathbb{D}}\{\min\nolimits_{\mathbb{D}}\left\{
f_{A}^{\mathbb{D}}(x_{1}),f_{A}^{\mathbb{D}}(x_{2})\right\} ,\min\nolimits_{%
\mathbb{D}}\left\{ f_{B}^{\mathbb{D}}(x_{1}),f_{B}^{\mathbb{D}%
}(x_{2})\right\} \}
\end{eqnarray*}

or equivalently%
\begin{eqnarray*}
&&f_{C}^{\mathbb{D}}\{\lambda x_{1}+(1-\lambda )x_{2}\} \\
&\succeq &\min\nolimits_{\mathbb{D}}\{\min\nolimits_{\mathbb{D}}\left\{
f_{A}^{\mathbb{D}}(x_{1}),f_{B}^{\mathbb{D}}(x_{1})\right\} ,\min\nolimits_{%
\mathbb{D}}\left\{ f_{A}^{\mathbb{D}}(x_{2}),f_{B}^{\mathbb{D}%
}(x_{2})\right\} \}
\end{eqnarray*}

and thus%
\[
f_{C}^{\mathbb{D}}\{\lambda x_{1}+(1-\lambda )x_{2}\}\succeq \min\nolimits_{%
\mathbb{D}}\left\{ f_{C}^{\mathbb{D}}(x_{1}),f_{C}^{\mathbb{D}%
}(x_{2})\right\} . 
\]
\end{proof}

\begin{theorem}
If two $\mathbb{D-}$fuzzy sets $A$ and $B$ on the Universal sets $X$ and $Y$
are convex then their cartesian product $A\times B$ is also convex on $%
X\times Y$.
\end{theorem}

\begin{proof}
Let $C=A\times B.$ Then%
\begin{eqnarray*}
f_{C}^{\mathbb{D}}\{\lambda (x_{1},y_{1})+(1-\lambda )(x_{2},y_{2})\}
&=&f_{C}^{\mathbb{D}}\{(\lambda x_{1}+(1-\lambda )x_{2},\lambda
y_{1}+(1-\lambda )y_{2})\} \\
&=&\min\nolimits_{\mathbb{D}}\{f_{A}^{\mathbb{D}}(\lambda x_{1}+(1-\lambda
)x_{2}),f_{B}^{\mathbb{D}}(\lambda y_{1}+(1-\lambda )y_{2})\}
\end{eqnarray*}

Now since $A$ and $B$ are convex%
\[
f_{A}^{\mathbb{D}}\{\lambda x_{1}+(1-\lambda )x_{2}\}\succeq \min\nolimits_{%
\mathbb{D}}\left\{ f_{A}^{\mathbb{D}}(x_{1}),f_{A}^{\mathbb{D}%
}(x_{2})\right\} 
\]%
\[
f_{B}^{\mathbb{D}}\{\lambda y_{1}+(1-\lambda )y_{2}\}\succeq \min\nolimits_{%
\mathbb{D}}\left\{ f_{B}^{\mathbb{D}}(y_{1}),f_{B}^{\mathbb{D}%
}(y_{2})\right\} . 
\]

Hence%
\begin{eqnarray*}
&&f_{C}^{\mathbb{D}}\{\lambda (x_{1},y_{1})+(1-\lambda )(x_{2},y_{2})\} \\
&\succeq &\min\nolimits_{\mathbb{D}}\{\min\nolimits_{\mathbb{D}}\left\{
f_{A}^{\mathbb{D}}(x_{1}),f_{A}^{\mathbb{D}}(x_{2})\right\} ,\min\nolimits_{%
\mathbb{D}}\left\{ f_{B}^{\mathbb{D}}(y_{1}),f_{B}^{\mathbb{D}%
}(y_{2})\right\} \} \\
&=&\min\nolimits_{\mathbb{D}}\{\min\nolimits_{\mathbb{D}}\left\{ f_{A}^{%
\mathbb{D}}(x_{1}),f_{B}^{\mathbb{D}}(y_{1})\right\} ,\min\nolimits_{\mathbb{%
D}}\left\{ f_{A}^{\mathbb{D}}(x_{2}),f_{B}^{\mathbb{D}}(y_{2})\right\} \}.
\end{eqnarray*}

Therefore%
\[
f_{C}^{\mathbb{D}}\{\lambda (x_{1},y_{1})+(1-\lambda )(x_{2},y_{2})\}\succeq
\min\nolimits_{\mathbb{D}}\{f_{C}^{\mathbb{D}}((x_{1},y_{1})),f_{C}^{\mathbb{%
D}}((x_{2},y_{2}))\}. 
\]
\end{proof}

\begin{remark}
\label{R1} A cylinder $P$ in $X\times Y$ is a convex $\mathbb{D-}$fuzzy set
since it is the cartesian product of two convex set.
\end{remark}

\begin{definition}
\label{Db} A $\mathbb{D-}$fuzzy set $A$ is bounded if and only if the sets $%
\Gamma _{\alpha }=\{x:f_{A}^{\mathbb{D}}(x)\succeq \alpha \}$ are $\mathbb{D-%
}$bounded for all $\alpha \succ 0;$ that is for every $\alpha \succ 0$ there
exists a finite $R(\alpha )$ such that $\left\Vert x\right\Vert _{\mathbb{D}%
}\preceq R(\alpha )$ for all $x$ in $\Gamma _{\alpha }.$
\end{definition}

Let $A$ be a $\mathbb{D-}$bounded set and for $\epsilon \succ 0$ consider $%
\Gamma _{\epsilon }=\{x:f_{A}^{\mathbb{D}}(x)\succeq \epsilon \}.$ Then by
the definition \ref{Db}, $\Gamma _{\epsilon }$ is contained in a sphere $S$
of radius $R(\epsilon ).$ Let $H$ be a supporting hyperplane of $S.$ Then,
all points on the side of $H$ which does not contain the origin lie outside
or on $S,$ and so for all such points $f_{A}^{\mathbb{D}}(x)\preceq \epsilon
.$ Hence we can say that, if $A$ is $\mathbb{D-}$bounded set, then for each $%
\epsilon \succ 0$ there exists a hyperplane $H$ such that $f_{A}^{\mathbb{D}%
}(x)\preceq \epsilon $ for all $x$ on the side of $H$ which does not contain
the origin.

\begin{theorem}
Let $A$ be a $\mathbb{D-}$bounded $\mathbb{D-}$fuzzy set and $M=\{\sup_{%
\mathbb{D}}f_{A}^{\mathbb{D}}(x):x\in X\}.$ Then there is atleast one point $%
x_{0}$ at which $M$ is essentially attained in the sense that, for each $%
\epsilon \succ 0,$ every spherical neighborhood of $x_{0}$ contains points
in the set $Q(\epsilon )=\{x:f_{A}^{\mathbb{D}}(x)\succeq M-\epsilon \}.$
\end{theorem}

\begin{proof}
Consider a nested sequence of $\mathbb{D-}$bounded sets $\Gamma _{1},\Gamma
_{2},...,$ where 
\[
\Gamma _{n}=\{x:f_{A}^{\mathbb{D}}(x)\succeq M-M/(n+1)\},n=1,2,... 
\]

Since $M=\{\sup_{\mathbb{D}}f_{A}^{\mathbb{D}}(x):x\in X\},$ $\Gamma _{n}$
is nonempty for all finite $n.$

Let $x_{n}$ be an arbitrarily chosen point in $\Gamma _{n},n=1,2,....$ Then $%
x_{1},x_{2},...,$ is a sequence of points in a closed $\mathbb{D-}$bounded
set $\Gamma _{1}.$ By the Bolzano-Weierstrass theorem, thus sequence must
have atleast one limit point, say $x_{0},$ in $\Gamma _{1}.$ Consequently,
every spherical neighborhood of $x_{0}$ will contain infinitely many points
from the sequence $x_{1},x_{2},...,$ and more particularly, from the
subsequence $x_{N+1},$ $x_{N+2},$ $...,$ where $N\geq M/\epsilon .$ Since
the points of this subsequence fall within the set $Q(\epsilon )=\{x:f_{A}^{%
\mathbb{D}}(x)\succeq M-\epsilon \},$ the theorem is proved.
\end{proof}

A $\mathbb{D-}$fuzzy set $A$ is said to be strictly convex if the sets $%
\Gamma _{\alpha },$ $0\prec \alpha \preceq 1$ are strictly convex. A $%
\mathbb{D-}$fuzzy set $A$ is said to be strongly convex if, for any two
distinct points $x_{1}$ and $x_{2}$, and any $\lambda \in (0,1)_{\mathbb{D}}$%
\[
f_{A}^{\mathbb{D}}\{\lambda x_{1}+(1-\lambda )x_{2}\}\succ \min\nolimits_{%
\mathbb{D}}\{f_{A}^{\mathbb{D}}(x_{1}),f_{A}^{\mathbb{D}}(x_{2})\}. 
\]

Note that strong convexity does not imply strict convexity or vice-versa. If 
$A$ and $B$ are strictly (strongly) convex, their intersection is strictly
(strongly) convex.

Let $A$ be a convex $\mathbb{D-}$fuzzy set and $M=\sup_{\mathbb{D}}f_{A}^{%
\mathbb{D}}(x).$ If $A$ is $\mathbb{D-}$bounded, then either $M$ is attained
for some $x,$ say $x_{0},$ or there is atleast one point $x_{0}$ at which $M$
is essentially attained in the sense that, for each $\epsilon \succ 0,$
every spherical neighborhood of $x_{0}$ contains points in the set $%
Q(\epsilon )=\{x:M-f_{A}^{\mathbb{D}}(x)\preceq \epsilon \}.$ Note that if $%
A $ is strongly convex and $x_{0}$ is attained, then $x_{0}$ is unique.

Let $A$ be a $\mathbb{D-}$fuzzy set with $\mathbb{D}$-membership function $%
f_{A}^{\mathbb{D}}(x),$ $x\in X$ and $M=\sup_{\mathbb{D}}f_{A}^{\mathbb{D}%
}(x).$ Then the set of all points in $X$ at which $M$ is essentially
attained, is called the core of $A$ and is denoted by $C(A).$

\begin{theorem}
If $A$ is a convex $\mathbb{D-}$fuzzy set, then $C(A)$ is convex.
\end{theorem}

\begin{proof}
It is sufficient to show that if $M$ is essentially attained at $x_{0}$ and $%
x_{1},$ $x_{0}\neq x_{1},$ then it is also essentially attained at all $%
x=\lambda x_{0}+(1-\lambda )x_{1},\lambda \in \lbrack 0,1]_{\mathbb{D}}.$

Let $P$ be a cylinder os radius $\epsilon _{\mathbb{D}}$ with the passing
through $x_{0}$ and $x_{1}$ as its axis. Let $x_{0}^{\prime }$ be a point in
a sphere of radius $\epsilon _{\mathbb{D}}$ centering on $x_{0}$ and $%
x_{1}^{\prime }$ be a point in a sphere of radius $\epsilon _{\mathbb{D}}$
centering on $x_{1}$ such that $f_{A}^{\mathbb{D}}(x_{0}^{\prime })\succeq
M-\epsilon _{\mathbb{D}}$ and $f_{A}^{\mathbb{D}}(x_{1}^{\prime })\succeq
M-\epsilon _{\mathbb{D}}.$ Then, by the convexity of $A$, for any point $u$
on the segment $x_{0}^{\prime }x_{1}^{\prime },$ we have $f_{A}^{\mathbb{D}%
}(u)\succeq M-\epsilon _{\mathbb{D}}.$ Furthermore, by the remark \ref{R1},
all points on $x_{0}^{\prime }x_{1}^{\prime }$ will lie on $P.$

Now let $x$ be any point in the segment $x_{0}x_{1}.$ Then for any point $y$
in the segment $x_{0}^{\prime }x_{1}^{\prime }$, we have 
\[
d_{\mathbb{D}}(x,y)\preceq \epsilon _{\mathbb{D}},\text{ since }%
x_{0}^{\prime }x_{1}^{\prime }\in P. 
\]%
where $d_{\mathbb{D}}$ is the hyperbolic valued metric.

Hence, a sphere of radius $\epsilon _{\mathbb{D}}$ centering on $x$ will
contain at least one point of the segment $x_{0}^{\prime }x_{1}^{\prime }$
and so will contain at least one point, say $w,$ at which $f_{A}^{\mathbb{D}%
}(w)\succeq M-\epsilon _{\mathbb{D}}.$ This implies that $M$ is essentially
attained at $x$ and thus proves the theorem.
\end{proof}

\begin{corollary}
If $A$ is strongly convex, then the point at which $M$ is essentially
attained is unique.
\end{corollary}

Let $A$ be a $\mathbb{D-}$fuzzy set in $E^{n}$ with membership function $%
f_{A}^{\mathbb{D}}(x)$ $=$ $f_{A}^{\mathbb{D}}(x_{1},$ $x_{2},$ $...,$ $%
x_{n}).$ The $\mathbb{D}$-shadow(or $\mathbb{D}$-projection) of $A$ on a
hyperplane $H=\{x:x_{1}=0\}$ is defined to be a $\mathbb{D}$-fuzzy set $%
S_{H}(A)$ in $E^{n-1}$ with $f_{S_{H}(A)}^{\mathbb{D}}(x)$ given by%
\[
f_{S_{H}(A)}^{\mathbb{D}}(x)=f_{S_{H}(A)}^{\mathbb{D}}(x_{2},...,x_{n})=\sup%
\nolimits_{x_{1}}f_{S_{H}(A)}^{\mathbb{D}}(x_{1},x_{2},...,x_{n}). 
\]

It is clear from the definition that, if $A$ is a convex $\mathbb{D-}$fuzzy
set, then its $\mathbb{D}$-shadow on any hyperplane is also a convex $%
\mathbb{D-}$fuzzy set.

\begin{theorem}
Let $A$ and $B$ be two convex $\mathbb{D-}$fuzzy sets in $E^{n}$. Then%
\[
S_{H}(A)=S_{H}(B)\text{ for all }H\Longrightarrow A=B. 
\]
\end{theorem}

\begin{proof}
It is sufficient to show that if there exists a point, say $x_{0},$ such
that $f_{A}^{\mathbb{D}}(x_{0})\neq f_{B}^{\mathbb{D}}(x_{0}),$ then there
exists a hyperplane $H$ such that $f_{S_{H}(A)}^{\mathbb{D}}(x_{0}^{\ast
})\neq f_{S_{H}(B)}^{\mathbb{D}}(x_{0}^{\ast })$, where $x_{0}^{\ast }$ is
the projection of $x_{0}$ on $H.$

Suppose that $f_{A}^{\mathbb{D}}(x_{0})=\alpha \succ f_{B}^{\mathbb{D}%
}(x_{0})=\beta .$ Since $B$ is a convex $\mathbb{D-}$fuzzy set, the set $%
\Gamma _{\beta }=\{x:f_{B}^{\mathbb{D}}(x)\succ \beta \}$ is convex, and
hence there exists a hyperplane $F$ supporting $\Gamma _{\beta }$ and
passing through $x_{0}.$ Let $H$ be a hyperplane orthogonal to $F$ and $%
x_{0}^{\ast }$ be the $\mathbb{D}$-projection of $x_{0}$ on $H.$ Then, since 
$f_{B}^{\mathbb{D}}(x)\preceq \beta $ for all $x$ on $F,$ we have $%
f_{S_{H}(B)}^{\mathbb{D}}(x_{0}^{\ast })\preceq \beta .$ On the other hand $%
f_{S_{H}(A)}^{\mathbb{D}}(x_{0}^{\ast })\succeq \alpha .$ Consequently, $%
f_{S_{H}(B)}^{\mathbb{D}}(x_{0}^{\ast })\neq f_{S_{H}(A)}^{\mathbb{D}%
}(x_{0}^{\ast }).$ Similarly we can show this for the case $\alpha \prec
\beta .$
\end{proof}

Let $A$ and $B$ are two $\mathbb{D-}$bounded $\mathbb{D-}$fuzzy sets.and let 
$H$ be a hypersurface in $E^{n}\,$\ defined by an equation $h(x)=0,$ with
all points for which $h(x)\geq 0$ being on one side of $H$ and all points
for which $h(x)\leq 0$ being on the other side. Let $K_{H}$ be a number
dependent on $H$ such that $f_{A}^{\mathbb{D}}(x)\preceq K_{H}$ on one side
of $H$ and $f_{B}^{\mathbb{D}}(x)\preceq K_{H}$ on the other side. Let $%
M_{H}=\inf_{\mathbb{D}}K_{H}$ and $D=1-M_{H}.$ Then $D$ is called the degree
of separation of $A$ and $B$ by $H.$

The main problem is to find a member of the family of hypersurfaces $%
\{H_{\lambda }\},$ $\lambda \in E^{m},$ which realizes the highest possible
degree of separation.

The following theorem is called the separation theorem for convex $\mathbb{D-%
}$fuzzy sets.

\begin{theorem}
Let $A$ and $B$ be two $\mathbb{D-}$bounded $\mathbb{D-}$fuzzy sets in $%
E^{n},$ with maximal $\mathbb{D}$-grades $M_{A}$ and $M_{B}$ respectively.
Let $M$ be the maximal $\mathbb{D}$-grade for the intersection $A\cap B.$
Then $D=1-M.$
\end{theorem}

\begin{proof}
The proof of this theorem is similar to the proof of the separation theorem
for convex ordinary fuzzy sets (\cite{Za}, page no- 352).
\end{proof}

\end{document}